\documentclass[11pt]{amsart}

\usepackage{amsmath,amssymb,amscd,amsfonts,verbatim}
\usepackage{paralist}
\usepackage[mathscr]{eucal}
\usepackage{hyperref}
\usepackage{color}

\newtheorem{thm}{Theorem}[section]
\newtheorem{lem}[thm]{Lemma}

\newtheorem{prop}[thm]{Proposition}
\newtheorem{cor}[thm]{Corollary}
\newtheorem{defn}[thm]{Definition}

\newtheorem{assu-nota}[thm]{Assumption--Notation}
\theoremstyle{remark}

\newcommand{\Z}{\mathbb Z}

\newcommand{\G}{\mathbb G}

\newcommand{\pp}{\mathbb P}

\newcommand{\OO}{\mathcal O}

\DeclareMathOperator{\Pic}{Pic}
\DeclareMathOperator{\NS}{NS}

\DeclareMathOperator{\im}{Im}

\newcommand{\Si}{\Sigma}

\newcommand{\fie}{\varphi}

\numberwithin{equation}{section}

\title{On surfaces of general type with $q=5$}
\author{Margarida Mendes Lopes, Rita Pardini and Gian Pietro Pirola}

\thanks{%{\it Mathematics Subject Classification (2000)}: 14J29. \\
The first author is a member of the Center for Mathematical
Analysis, Geometry and Dynamical Systems (IST/UTL).  The second and the third author are members of G.N.S.A.G.A.--I.N.d.A.M.}
\begin{document}
\begin{abstract}
We prove that a complex surface $S$ with irregularity $q(S)=5$ that has no irrational  pencil of genus $>1$ has geometric genus $p_g(S)\ge 8$. As a consequence, one is  able to classify minimal surfaces $S$ of general type with $q(S)=5$ and $p_g(S)<8$.
This result  is a negative answer, for $q=5$,  to the question asked in \cite{mr} of the existence of surfaces of general type  with irregularity $q\ge 4$  that have no irrational pencil of genus $>1$ and with the lowest possible geometric  genus $p_g=2q-3$. This   gives some evidence for the conjecture 
that the only  irregular surface with no irrational pencil of genus $>1$ and  $p_g=2q-3$ is the symmetric product of a  genus three curve. \medskip

\end{abstract}
\maketitle
%\tableofcontents
\section{Introduction} 

Let $S$ be a smooth complex projective surface with irregularity $q(S):=h^0(\Omega^1_S)\ge 3$. The existence of a fibration $f\colon S\to B$ with $B$ a smooth curve of genus $b>1$ (``an irrational pencil of genus $b>1$'') gives much geometrical  information on $S$ (cf. the survey  \cite{survey}).  However,  surfaces with an irrational pencil of genus $b>1$ can hardly  be regarded as ``general'' among the irregular surfaces of general type:  for instance, for $b<q(S)$  the Albanese variety of such a surface $S$ is not simple.

By the classical  Castelnuovo-De Franchis theorem, if $S$ has no irrational pencil of genus $>1$ then the inequality $p_g(S)\ge 2q(S)-3$ holds, where $p_g(S):=h^0(K_S)$ is, as usual,  the geometric genus.  Note that this inequality has been recently generalized in \cite{PareschiPopa} to K\"ahler varieties of arbitrary dimension. 

The surfaces of general type  $S$  for which the equality  $p_g(S)=2q(S)-3$ holds are  studied in \cite{mr}.  There  those with an irrational pencil of genus $>1$ are classified and the inequality $K^2_S\ge 7\chi(S)-1$ is proven for $S$ minimal.  However, the question of the existence of surfaces with $p_g(S)=2q(S)-3$ having no irrational pencil of genus $b>1$ is  widely open. At present, the state of the art is as follows:
\begin{itemize}
\item for $q=3$, the only such surfaces are (the minimal desingularization of) a theta divisor in a principally polarized abelian threefold (\cite{HaconPardini}, \cite{Pirola});
\item for $q=4$, if $S$ is minimal then $K^2_S=16,17$ (\cite{bnp},\cite{cau});
\item for $q\ge 4$, no example is known.
\end{itemize}

\medskip

One is led to conjecture that the only  irregular surface with no irrational pencil of genus $>1$ and  $p_g=2q-3$ is the symmetric product of a  genus three curve.  In this note we settle the case $q=5$: 
\begin{thm}\label{noq5}   Let $S$ be a smooth projective complex surface  with $q(S)=5$ 
 that has  no irrational pencils of genus $>1$. Then:
 $$p_g(S)\geq 8.$$
\end{thm}

As a consequence we obtain the following classification theorem:
\begin{thm}\label{irrpencil} Let  $S$ be a minimal complex surface of general type with $q(S)=5$ and $p_g(S)\le 7$. 
Then either:
\begin{enumerate}
\item  $p_g(S)=6$, $K_S^2=16$ and $S$ is the product of a curve of genus 2 and a curve of genus 3; or

\item $p_g(S)=7$, $K_S^2=24$ and  $S=(C\times F)/\Z_2$, where $C$  is a curve of  genus $7$ with a free $\Z_2$-action, $F$  is a curve of genus  2 with a $\Z_2$-action such that $F/\Z_2$ has genus 1 and $\Z_2$ acts diagonally on $C\times F$.  The  map $f\colon S \to C/\Z_2$  induced by the projection $C\times F\to C$ is an irrational pencil  of genus   $4$ with general fibre $F$ of genus 2. 
\end{enumerate}

\end{thm}

The idea of the proof of Theorem \ref{noq5}  is to obtain contradictory upper and  lower bounds for $K^2_S$ under the assumption that $p_g(S)<8$ and $S$ is minimal.

For fixed $q$ and $p_g$, by  Noether's formula  giving an upper bound for $K^2$ is the same as giving a lower  bound for the topological Euler characteristic $c_2$. More precisely, it is the same as giving a lower bound for $h^{1,1}$, the only Hodge  number which is not determined by $p_g$ and $q$. 
 In our situation, the upper bound follows directly  from the result of \cite{cau}  that  if  $S$ is a  surface  of general type with $q=5$, having  no irrational pencils, then  $h^{1,1}\geq  11+t$, where $t$ is  bigger or equal to the number of curves contracted by the Albanese map. 
 
If the canonical system $|K_S|$ has no fixed components, one can apply the results of \cite{bnp} to get a lower bound for $K^2_S$ which is enough to rule out this possibility. Hence the bulk of the proof consists in obtaining a lower bound for $K^2_S$ under the assumption that $|K_S|$ has a fixed part $Z>0$. This is done in \S \ref{sec:reider}, where we improve by 1 in the case $Z>0$ a well known inequality  for surfaces with birational bicanonical map due to Debarre  (cf. Corollary \ref{canonical}). The proof is based on a subtle numerical analysis of the intersection properties of the fixed and moving part of $|K_S|$  that is,  we believe,   of independent interest.
\bigskip

 It would be possible to generalize   Theorem  \ref{noq5} for $q\geq 6$, if a good lower bound for   $h^{1,1}(S)$ could be established.   Unfortunately it is very difficult to extend the  methods of  \cite{cau}  for $q\geq 6$.  Recently, a lower bound on $h^{1,1}$ has been obtained   in \cite{lapo} by completely different methods, but   it is not  strong enough for our purposes.
\smallskip

\noindent {\em Acknowledgments:\/}  This research was partially supported by  FCT (Portugal) through program POCTI/FEDER and Project 
PTDC/MAT/099275/2008 and by MIUR (Italy) through project   PRIN 2007 \emph{``Spazi di moduli e teorie di Lie''}.

We wish to thank Letterio Gatto and Stavros Papadakis  for very helpful  conversations.

  \bigskip
\noindent{\bf Notation and conventions:} a {\em surface} is  a smooth complex projective surface. We use the standard notation for the invariants of a surface $S$: $p_g(S):=h^0(\omega_S)=h^2(\OO_S)$ is the {\em geometric genus}, $q(S):=h^0(\Omega^1_S)=h^1(\OO_S)$ is the {\em irregularity} and $\chi(S):=p_g(S)-q(S)+1$ is the {\em Euler--Poincar\'e} characteristic.

An {\em irrational pencil of genus $b$} of a surface $S$ is a fibration $f\colon S\to B$, where $B$ is a smooth curve of genus $b>0$.

We use $\equiv$ to denote linear equivalence and $\sim$ to denote numerical equivalence of divisors.

\section{Reider divisors} \label{sec:reider}

Let $S$ be a surface and let  $M$ be a nef and big divisor  on $S$ such that $M^2\geq 5$. By Reider's theorem,  if a point $P$ of $S$ is a base point of $|K_S+M|$, then there is an effective divisor $E$ passing through $P$ 
such that either:
\begin{itemize}  \item $E^2=-1$, $ME=0$ or \item $E^2=0$, $ME=1$.\end{itemize} 
This suggests the following definition: 
\begin{defn}  Let $M$ be a nef and big divisor  on a surface $S$.  An effective  divisor  $E$  such that $E^{2}=k $ and $EM=s$ is called a $(k,s)$ divisor of $M$. 
\end{defn}
By \cite[(0.13)]{ccm}, the  $(-1,0)$ divisors and the $(0,1)$ divisors are $1$-connected. 
% Furthermore if $h^0(M)$ is big enough (i.e., if  $h^0(M)\geq 2$ for  $ME=0$ and   $h^0(M)\geq 3$ for $ME=1$) then $h^0(M-E)\neq 0$.  
 In addition, if $E$ is a $(-1,0)$ divisor, using the index theorem one shows that the intersection form on the components of $E$  is negative definite. In particular, there exist only finitely many $(-1,0)$ divisors of $M$ on $S$.

  \begin{lem}\label{lem:01divisor} Let $M$ be a nef and big divisor  on a projective surface $S$.   Then:
  \begin{enumerate}
   \item if $C$ is an irreducible component of a  reducible $(0,1)$ divisor $E$  of $M$, then $C^2<0$;
  \item if  $E_1,E_2$ are two distinct $(0,1)$  divisors  of $M$,  then  
$E_1E_2= 0$ and  $E_1$ and $E_2$ are disjoint.
\end{enumerate}
\end{lem}
\begin{proof}
Let $E$, $C$ be as in (i). The index theorem gives   $C^2<0$ if $MC=0$ and $C^2\le 0$ if $MC=1$. Assume that $C^2=0$. Then $EC=(E-C)C>0$, since $E$ is $1$-connected, and therefore $(E+C)^2\ge 2$. Since $M^2\geq 5$  and $M(C+E)=2$ we have a contradiction to the index theorem.  Hence $C^2<0$.

Next we prove (ii). We have:
  $$M^2\geq 5,\quad M(E_1+E_2)=2, \quad M(E_1-E_2)=0,$$
  hence  by the index theorem we obtain: 
  $$2E_1E_2=(E_1+E_2)^2\leq 0, \quad -2E_1E_2=(E_1-E_2)^2\le 0.$$
So $E_1E_2=0$.
By 1-connectedness of $E_1$, $E_2$ we conclude that  neither divisor is contained in the other.  Then we can   write  $E_1=A+B$, $E_2=A+C$
where $A\geq 0$, $B, C>0$  and $B$ and $C$ have no common components.

Since $M$ is nef and $ME_i=1$, we have $1\geq MB(=MC)$ and so $B^2\leq 0, C^2\leq 0$.   
Then, since $0=(E_1-E_2)^2=(B-C)^2$, we conclude that $B^2=C^2=BC=0$. Hence $B$ and $C$ are disjoint, $MB=MC=1$ and $B$ is numerically equivalent to $C$.   Since $B$ is also a $(0,1)$ divisor, $BE_1=0$ and so, by 1-connectedness of $E_1$ we conclude that $A=0$.
\end{proof}

\begin{lem}\label{rational} Let $S$ be a surface and let $M$ be a nef and big divisor such that  the linear system $|M|$ has no fixed components.
  Let $E$ be a  $(0,1)$  divisor  of $M$ and let $C$ be the only  irreducible component of $E$ such that $MC=1$.
Then either   $|M|$ has a  base point on $C$ or  $C$ is a smooth rational curve.
\end{lem}

\begin{proof}  Suppose $|M|$ has no base points on $C$. 
Then, since $MC=1$ the  restriction map $H^0(M)\to H^0(C, M)$ has image of dimension at least $2$. It follows that $C$ is a smooth rational curve.
\end{proof} 
 \begin{prop}\label{fixed}Let  $X$ be a non ruled surface and let $M$ be a divisor of $X$ such that:
 \begin{itemize}
 \item $M^2\geq 5$, 
   \item the system $|M|$ has no fixed component and maps $X$ onto a surface.
\end{itemize}
Let  $C$ be  an irreducible curve contained in the fixed locus of $|K_X+M|$.  Then either:
\begin{enumerate}
\item $C$ is contained in a $(-1,0)$ divisor of $M$, $MC=0$ and $C^2<0$; 

 or
\item $C$ is contained in a $(0,1)$ divisor of $M$, $MC\leq 1$ and $C^2\leq 0$.
\end{enumerate}
 \end{prop}
\begin{proof}   

Let $P\in C$ be a point.   By Reider's theorem, there is a $(-1,0)$ divisor or a $(0,1)$ divisor of $M$ passing through  $P$. 

Assume for contradiction   that $C$ is not a component of any $(-1,0)$ or $(0,1)$ divisor of $M$.  Since there are  only  finitely  many  distinct $(-1,0)$ divisors of $M$ in $S$,  we can assume that there is a $(0,1)$ divisor passing through a general point $P$ of $C$.   
It follows that there are infinitely many  $(0,1)$ divisors on $S$. Recall that two distinct $(0,1)$ divisors are disjoint  by Lemma \ref{lem:01divisor}.   Thus, since $|M|$ has a finite number of base points, by Lemma \ref{rational} $X$ is  ruled, against the assumptions.

So $C$  is contained in a  $(-1,0)$ divisor or a $(0,1)$ divisor $E$ of $M$. In the first case, $M$ being nef implies that $MC=0$ and so  $C^2<0$ by the index theorem. In the second case, again by nefness $MC\leq 1$ and again by the index theorem $C^2\leq 0$.
\end{proof}

 \begin{lem}\label{EL}
  Let $S$ be a surface and let  $M$ be a nef and big divisor  of  $S$ and let $E$ be a $(0,1)$ divisor of $M$.  If $L$ is a  divisor such that $(M-L)^2>0$ and $M(M-L)>0$, then $EL\leq 0$.
  \end{lem}
\begin{proof}   

Write $\gamma:=M(M-L)$. Then  $M(\gamma E-(M-L))=0$. Since   $(M-L)^2>0$ and $E^2=0$, $\gamma (M-L)\not\sim E$. Thus, by the index theorem $0>  (\gamma E-(M-L))^2= -2\gamma E(M-L)+(M-L)^2$.

So $E(M-L)>0$, and therefore $EL\leq 0$.\end{proof}

\begin{prop}\label{2M} Let $S$ be a smooth  minimal surface of general type and let $M$ be a  divisor such that 
 \begin{itemize}  
 \item  $Z:=K_S-M>0$;
 \item the linear system $|M|$ has no fixed components and maps $S$ onto a surface. 
 \end{itemize} 
 Then the following hold:
 \begin{enumerate}
 \item if  $M^2\geq 5+KZ$,  then $h^0(2M)<h^0(K_S+M)$;
 \item if   $M^2\geq 5$,  $(M-Z)^2>0$ and $M(M-Z)>0$, then  there are no $(0,1)$ divisors of $M$.   Furthermore $h^0(2M)<h^0(K_S+M)$ and every irreducible fixed component $C$ of $|K_S+M|$ satisfies  $MC=0$.
 \end{enumerate}\end{prop}
\begin{proof}
We observe first of all that  $h^{0}(2M)=h^{0}(K_S+M)$ if and only if $Z$ is the fixed part of $|K_S+M|$.

 {\rm (i)} Assume for  contradiction that
 $h^{0}(2M)=h^{0}(K_S+M)$. %This implies that $Z$ is the fixed divisor of $|M+K_S|$.
 Let $C$ be an irreducible component of $Z$. By Proposition  \ref{fixed},  
$C^{2}\leq 0$  and $MC\leq 1.$  Now $$-2\leq C^{2}+KC\leq C^{2}+KZ,$$ and hence $C^{2}\geq -2-KZ.$
It follows
$$(M-C)^{2}=M^{2}-2MC +C^2\geq  M^{2}-2 -2 -KZ=M^2-4-KZ>0.$$ 
In addition, we have:
$$M(M-C)=(M-C)^2+C(M-C)\ge (M-C)^2-C^2\ge (M-C)^2>0.$$
 Since  $MZ\geq 2$ by the  2-connectedness of canonical  divisors,   there is at least a component $D$ of $Z$ such that $MD>0$.  By Proposition \ref{fixed}, we have  $MD=1$ and $D$ is contained in  a $(0,1)$ divisor $E$ of $M$.
Then  Lemma \ref{EL} gives  $EC\le 0$  for all the components of $Z$,  and so  $EZ\le 0.$

But now since $ME=1$ and $E^{2}=0$ we obtain that
$KE=1+EZ\le 1$. On the other hand, $K_SE$ is $>0$ by the index theorem and it is even by the adjunction formula, hence we have   a contradiction.
\medskip

   {\rm (ii)} Let $E$ be a $(0,1)$ divisor of $M$. Then   we have $EZ\le 0$ by Lemma \ref{EL} and we get a contradiction as above. So there are no $(0,1)$ divisors of $M$ on $S$.
 Hence by  Proposition \ref{fixed} every irreducible fixed curve of $|K_S+M|$ satisfies $MC=0$. Since $MZ\geq 2$ by the  2-connectedness of the canonical divisors, not every component of $Z$ can be a fixed component of $|K_S+M|$ and therefore $h^0(K_S+M)>h^0(2M)$.
\end{proof}

As  a consequence, we obtain the following refinement of Thm. 3.2 and Rem. 3.3 of \cite{deb1}:

\begin{cor}\label{canonical}  Let $S$ be a minimal surface of general type whose  canonical map is not composed with a pencil.   Denote by $M$ the moving part and by $Z$ the fixed part of $|K_S|$. If  $Z>0$ and $M^2\geq 5+K_SZ$, then 
$$K_S^2+\chi(S)=h^0(K_S+M)+K_SZ+MZ/2 \geq h^0(2M)+K_SZ+MZ/2+1.$$

Furthermore,  if  $h^0(K_S+M)=h^0(2M)+1$ then  $|K_S+M|$ has base points and  there is a $(-1,0)$ divisor or a $(0,1)$ divisor $E$  of $M$ such that $EZ\geq 1$.
  \end{cor}

\begin{proof}   Since $M$ is nef and big, by Kawamata-Viehweg vanishing $h^0(K_S+M)=\chi(K_S+M)$, hence the equality   follows by the Riemann-Roch theorem whilst 
the inequality   is  Proposition \ref{2M},  (i).     

For the second assertion it suffices to notice that $h^0(K_S+M)=h^0(2M)+1$  means that the image of the restriction map $H^0(K_S+M)\to H^0(Z, (K_S+M)|_Z)$ is 1-dimensional. Since $(K_S+M) Z\geq 2$, the system $|K_S+M|$ has necessarily base points.  Thus there is a $(-1,0)$ divisor or a $(0,1)$ divisor $E$ of $M$.  By adjunction $K_SE\equiv E^2$(mod 2)   and so necessarily $EZ\geq 1$.  \end{proof}

\section{Proofs of Theorem \ref{noq5} and Theorem \ref{irrpencil}}
\begin{proof}[Proof of Theorem \ref{noq5}]
Let $a\colon S\to A$ be the Albanese map of $S$. Notice that by the classification of surfaces the assumptions that  $q(S)=5$ and $S$ has no irrational pencil of genus $>1$ imply  that $S$ is of general type and $a$ is generically finite onto its image. Without loss of generality we may assume that $S$ is minimal. By \cite{appendix},  an irregular surface of general type  having no irrational pencils of genus $>1$  satisfies  $p_g\geq 2q-3$.  We assume for contradiction that $p_g(S)=7=2q(S)-3$, so that  $\chi(S)=3$.  We denote by $\fie_K\colon S\to \pp^7$ the canonical map and by $\Si$ the canonical image. Since $q(S)>2$, $\Si$ is a surface by \cite{xiaoirreg}.

We denote by $t$ the rank of the  cokernel of the 
map
$a^{\ast}\colon \NS(A) \to \NS (S)$. Note that $t$ is bigger than  or equal to the number of irreducible curves contracted by the Albanese map.

Denote as usual by  $b_i(S)$ the $i$-th Betti number and by $c_2(S)$ the second Chern class of $S$.
 By \cite[Thm.1,(3)]{cau}, we have $b_2(S)\ge 31+t$, namely  $c_2(S)\ge 13+t$. By Noether's formula this is equivalent to:
 \begin{equation}\label{causin}
  K^2_S\le 23-t
  \end{equation}

 Denote by $\G$ the   Grassmannian of $2$-planes of $H^0(\Omega^1_S)^{\vee}$ and by $\G$ the Grassmannian of $2$-planes in $H^0(\Omega^1_S)$. 
By the Castelnuovo--De Franchis theorem, the kernel of the map 
 $\rho\colon \bigwedge ^2H^{0}(\Omega^1_S) \to H^{0}(K_S)$ does not contain any  nonzero simple tensor. Hence $\rho$ induces a morphism $\G^{\vee}\to\pp(H^0(K_S))$ which is finite onto its image. Since $\dim\G^{\vee}=6$,  it follows that $\ker \rho$ has dimension $3$, $\rho$ is surjective and it induces a finite map $\G^{\vee}\to \pp(H^0(K_S))$.  As a consequence, we have the following facts:
 \begin{itemize}
\item[(a)]  the surface $S$ is generalized Lagrangian, namely there exist independent $1$-forms  $\eta_1,\dots \eta_4\in H^0(\Omega^1_S)$ such that $\eta_1\wedge\eta_2+\eta_3\wedge\eta_4=0$. In addition, we may assume that $\eta_1\wedge \eta_2$ is a general $2$-form of $S$. In that case,  the fixed part of the linear system $\pp(\wedge^2V)$, where $V=<\eta_1, \dots \eta_4>$,  coincides with the fixed part of the canonical divisor (cf. \cite[\S 3]{severi}) .
\item[(b)] the canonical image $\Si$ is contained in  the intersection of $\G$ with the codimension $3$ subspace $T=\pp(\im \rho^t) \subset \pp^9=\pp(\bigwedge ^2H^{0}(\Omega^1_S))$,
 \item[(c)] since  $\G^{\vee}$ is the dual variety of $\G$, the space $T$ is not contained in an hyperplane tangent to $\G$, hence $Y:=\G\cap T$ is a smooth threefold.
 \end{itemize}
 
  Using Lefschetz hyperplane section theorem we see that $\Pic(Y)$ is   generated by the class of a hyperplane.
Then $\Si$ is  the scheme theoretic intersection of $Y$ with a hypersurface 
of degree $m\geq 2$  of $\pp^6$. Thus, since  $\G$ has degree 5  (cf.    \cite[Cor.1.11]{mukai}), it follows that 
$\deg\Si=5m$ and $\omega_{\Si}=\OO_{\Si}(m-2)$.
 By the proof of Thm. 1.2 of  \cite{mr},  the degree $d$ of  $\fie_K$ is different from $2$.  Since $K^2_S\le 23$ by \eqref{causin}, the inequality $K^2_S\ge d\deg\Si=5dm$ gives  $d=1$, namely  $\fie_K$ is birational onto its image. So we have  $m\ge 3$,  since  $\omega_{\G}=\OO_{\G}(-5)$ (cf. \cite[ Prop. 1.9]{mukai}) and  $\Si$ is of general type.
 % In particular, $M^2\ge 15$.

Write $|K_S|=|M|+Z$, where $Z$ is the fixed part and $M$ is the moving part. 
If $Z=0$, then in view of (a) we have   $K_S^2\geq 8\chi=24$ by \cite[Thm.1.2]{bnp}. This would contradict \eqref{causin}, hence $Z>0$. 

 Since $m>2$, every quadric that 
 contains $\Si$ must contain $Y$. Recall that $Y$ is obtained  from  $\G$ by intersecting with  3 independent
linear sections. Denote by $R$ the homogeneous coordinate ring of
$\G$. Since  $R$ is Cohen--Macaulay and  $Y$ has codimension $3$
in  $\G$,  these 3 linear sections form an
$R$-regular sequence. As a consequence (cf. \cite[Prop.1.1.5]{bruns})  the    (vector) dimension   of the space of quadrics of $\pp^6$ containing  $Y$ is the same as  same as the     (vector) dimension of the space of quadrics of $\pp^9$ 
containg $\G$.
 Since the latter  dimension is $5$ (cf. \cite[Prop.1.2]{mukai}), it follows that:
 $$h^0(2M)\geq h^0(\OO_{\pp^6}(2))-5=23.$$
 
Then by  \eqref{causin}  and  Corollary \ref{canonical} we have:
 \begin{equation}\label{eq:23}
 26-t\geq K_S^2+\chi(S)=h^0(K_S+M)+K_SZ+MZ/2 \geq 23+K_SZ+MZ/2+1.
 \end{equation}
So $K_SZ+MZ/2\leq 2-t$.   Recall that $MZ\ge 2$ by the $2$-connectedness of canonical divisors. 
  
Assume $K_SZ=0$. Then every component of $Z$ is an irreducible smooth rational curve with self-intersection $-2$ and as such it is  contracted by the Albanese map.
Since $K_SZ+MZ/2\leq 2-t$, the only possibility is 
 $t=1$ and $MZ=2$.  Hence $Z=rA$, where $A$ is a $-2$-curve. Since $MZ=2$  and $K_SZ=0$,  we have $Z^2=-2$ and so $r=1$. Hence $Z$ is a $-2$-cycle of type $A_1$. Then, again by (a) and  \cite[Thm.12]{bnp}, we get  $K^2\geq 8\chi=24$,  a contradiction.

 So $K_SZ>0$. Then by \eqref{eq:23} necessarily $K_SZ=1$, $MZ=2$ (yielding  $Z^2=-1$) and $h^0(K_S+M)= 23=h^0(2M)+1$.  Then by Corollary \ref{canonical}, there is a $(-1,0)$ or a $(0,1)$ divisor $E$ of $M$, and, since the hypotheses of Proposition \ref{2M}, (ii) are satisfied, $E$ must be a $(-1,0)$ divisor of $M$.
 
 Then $M(E+Z)=2$ and so by the algebraic index theorem $M^2(E+Z)^2- 4\leq 0$, yielding $(E+Z)^2\leq 0$. Since  $(E+Z)^2=-2+2EZ$  and, by   Corollary \ref{canonical}, $EZ\geq 1$, the only possibility is $EZ=1$ and $(E+Z)^2=0$.   In this case $K_S(E+Z) =2$ and this is impossible
 by \cite[Proposition 8.2]{bnp}, where it is shown that a minimal irregular surface  with $q\geq 4$, having no  irrational pencils of genus $>1$, cannot have effective divisors of arithmetic genus 2 and self-intersection $0$. 
    \end{proof}
    
    \begin{proof}[Proof of Theorem \ref{irrpencil}] By \cite{appendix}, a surface of general type $S$  with $q(S)=5$ has $p_g(S)\ge 6$ and, in addition, if $p_g(S)=6$ then $S$ is the product of a curve of genus $C$ and a curve of genus 3.
   Now statement (ii) is a consequence of Theorem \ref{noq5} and \cite[Thm.1.1]{mr}.
    \end{proof}

\bigskip

\begin{minipage}{13.0cm}
\parbox[t]{6.5cm}{Margarida Mendes Lopes\\
Departamento de  Matem\'atica\\
Instituto Superior T\'ecnico\\
Universidade T{\'e}cnica de Lisboa\\
Av.~Rovisco Pais\\
1049-001 Lisboa, PORTUGAL\\
mmlopes@math.ist.utl.pt
 } \hfill
\parbox[t]{5.5cm}{Rita Pardini\\
Dipartimento di Matematica\\
Universit\`a di Pisa\\
Largo B. Pontecorvo, 5\\
56127 Pisa, Italy\\
pardini@dm.unipi.it}

\vskip1.0truecm

\parbox[t]{5.5cm}{Gian Pietro Pirola\\
Dipartimento di Matematica\\
Universit\`a di Pavia\\
Via Ferrata, 1 \\
 27100 Pavia, Italy\\
\email{gianpietro.pirola@unipv.it}}
\end{minipage}

\end{document}